\documentclass[a4paper,reqno]{amsart}

\usepackage{tikz-cd}
\usepackage{amsmath}
\usepackage{amssymb}
\usepackage{enumerate}
\usepackage{ascmac}
\usepackage{mathrsfs}
\usepackage{multirow}
\usepackage{amsthm}
\usepackage{wrapfig}
\usepackage{here}
\usepackage{multirow}
\usepackage{hyperref}
\usepackage{graphicx,color}
\usepackage{caption}

\theoremstyle{plain}
    \newtheorem{thm}{Theorem}[section]
    \newtheorem{lem}[thm]{Lemma}
    \newtheorem{prop}[thm]{Proposition}
    \newtheorem{cor}[thm]{Corollary}
    
    \newtheorem{conje}[thm]{Conjecture}
    
    \newtheorem{claim}[thm]{Claim}
\theoremstyle{definition}    
    \newtheorem{defn}[thm]{Definition}
    
    \newtheorem{rem}[thm]{Remark}

\def\an{{\mathrm{an}}}

\def\Aut{{\mathrm{Aut}}}

\def\C{{\mathbb{C}}}

\def\CH{{\mathrm{CH}}}

\def\dec{{\mathrm{dec}}}

\def\dim{{\mathrm{dim}}}

\def\div{{\mathrm{div}}}

\def\id{{\mathrm{id}}}
\def\ind{{\mathrm{ind}}}

\def\P{{\mathbb{P}}}

\def\Q{{\mathbb{Q}}}

\def\Spec{{\mathrm{Spec}\,}}

\def\Z{{\mathbb{Z}}}

\begin{document}
\title{Notes on symplectic action on $(2,1)$-cycles on $K3$ surfaces}
\author{Ken Sato}
\address{Department of Mathematics, Institute of Science Tokyo}
\email{sato.k.da@m.titech.ac.jp}
\begin{abstract}
In this paper, we propose and study a conjecture that symplectic automorphisms of a $K3$ surface $X$ act trivially on the indecomposable part $\CH^2(X,1)_\ind\otimes \Q$ of Bloch's higher Chow group.
This is a higher Chow analogue of Huybrechts' conjecture on the symplectic action on $0$-cycles.
We give several partial results verifying our conjecture, some conditional and some unconditional.
Our unconditional results include the full proof for Kummer surfaces of product type.
\end{abstract}
\maketitle
\section{Introduction}
The higher Chow groups $\CH^p(X,q)$ of a smooth variety $X$, introduced by Bloch, are a generalization of the classical Chow groups $\CH^p(X)$.
They are related to many important invariants in algebraic geometry, $K$-theory, and number theory.
However, as with the classical Chow groups,  their structure remains mysterious when the codimension $p$ is greater than 1.

In this paper, we focus on the higher Chow group $\CH^2(X,1)$ of $K3$ surfaces.
For the classical Chow group $\CH^2(X)$ of $K3$ surfaces, there has been extensive work motivated by Bloch's conjecture, although its explicit structure remains out of reach.
In particular, regarding the actions of automorphisms, Huybrechts proposed the following conjecture \cite[Conjecture 3.4]{Huyconj}.
\begin{conje}\label{Huybrechts}
Let $X$ be a $K3$ surface and $\Aut_s(X)$ be the group of symplectic automorphisms of $X$.
Then $\Aut_s(X)$ acts trivially on $\CH^2(X)$.
\end{conje}
An automorphism of a $K3$ surface is called \textit{symplectic} if it acts trivially on a non-vanishing 2-form.
Conjecture \ref{Huybrechts} was proved in cases where $\Aut_s(X)$ is generated by elements of finite order (\cite{Voi}, \cite{Huysymp}).

This paper proposes the following analogue of Conjecture \ref{Huybrechts} and provides some supporting evidence for it.
\begin{conje}\label{main}
For a $K3$ surface $X$, $\Aut_s(X)$ acts trivially on $\CH^2(X,1)_\ind\otimes \Q$.
\end{conje}
Here, $\CH^2(X,1)_{\ind}$ is the \textit{indecomposable part} of $\CH^2(X,1)$, which is defined as the cokernel of the map $\mathrm{Pic}(X)\otimes \C^\times \to \CH^2(X,1)$ induced by the intersection product.
Since the $\Aut_s(X)$-action on $\mathrm{Pic}(X)$ is non-trivial, passing to the indecomposable part is essential.
We also need to restrict our attention to symplectic automorphisms because the action of non-symplectic automorphisms is known to be non-trivial in some cases.
The latter fact is used in the construction of non-trivial elements of $\CH^2(X,1)_\ind$ (e.g., \cite{sato}).

The relationship between Conjecture \ref{Huybrechts} and Conjecture \ref{main} can be explained in termes of motives as follows.
For a $K3$ surface $X$, let $t_2(X)$ be the \textit{transcendental part of the Chow motive} as defined in \cite{KMP}.
By the result of \cite[Theorem 2]{Kah}, we have the following natural isomorphisms.
\begin{equation}
\begin{aligned}
& \CH^2(X)\otimes \Q \simeq H^4_{\mathcal{M}}(t_2(X),\Q(2)) \\
& \CH^2(X,1)_\ind \otimes \Q \simeq H^3_{\mathcal{M}}(t_2(X),\Q(2))
\end{aligned}
\end{equation}
In other words, $\CH^2(X)\otimes \Q$ and $\CH^2(X,1)_\ind \otimes \Q$ arise from the same motives.
If Ayoub's conservativity conjecture \cite{ayoub} holds, symplectic automorphisms of finite order act trivially on $t_2(X)$, thus Conjecture \ref{Huybrechts} and Conjecture \ref{main} follow when $\Aut_s(X)$ is generated by elements of finite orders.

The bulk of this paper is devoted to producing unconditional results on Conjecture \ref{main}.
The typical examples of symplectic automorphisms are translations by sections of elliptic fibrations.
For a $K3$ surface $X$, let $\mathrm{MW}_{\mathrm{tor}}$ be the subgroup of $\Aut_s(X)$ generated by translations by torsion sections of all elliptic fibration structures\footnote{
Note that a $K3$ surface often has several different elliptic fibration structures.
}  on $X$.
In this paper, we prove the following.

\begin{thm}[Theorem \ref{newmainthm}]\label{mainthmintro}
The subgroup $\mathrm{MW}_{\mathrm{tor}}$ acts trivially on $\CH^2(X,1)_{\mathrm{ind}}\otimes \Q$.
In particular, for a $K3$ surface such that $\mathrm{MW}_\mathrm{tor}=\Aut_s(X)$, Conjecture \ref{main} holds.
\end{thm}

By Theorem \ref{mainthmintro},  we confirm Conjecture \ref{main} in the following case.

\begin{cor}[Corollary \ref{Kummer}]
For non-isogenus generic elliptic curves $E,F$, let $\mathrm{Km}(E\times F)$ be the Kummer surface associated with the product $E\times F$.
Then Conjecture\ref{main} is true for $\mathrm{Km}(E\times F)$.
\end{cor}

The proof of Theorem \ref{mainthmintro} is divided into three steps.
First, we show that the translation acts trivially on cycles supported on fibers, using Kodaira's classification of singular fibers \cite{Kod}.
In the second step, by constructing symbols in the Milnor $K_2$-group explicitly, we prove that the translation acts trivially on the cycles supported on sections, modulo cycles supported on fibers.
Finally, by a base change argument, we reduce the proof of Theorem \ref{mainthmintro} to the previous two cases.
Since we use the base change argument in the proof, our result in fact holds for elliptic surfaces that are not necessarily $K3$ surfaces (Proposition \ref{mainthm}).

Conjecture \ref{main} is also related to the injectivity of the \textit{transcendental regulator map}
\begin{equation}\label{transregintro}
\CH^2(X,1)_\ind\to  J(T(X)^\vee).
\end{equation}
In particular, for a $K3$ surface $X$ such that \eqref{transregintro} is injective after tensoring $\Q$, Conjecture \ref{main} holds (Proposition \ref{BH}).
The injectivity of \eqref{transregintro} after tensoring $\Q$ follows from the \textit{amended version of Beilinson's Hodge conjecture} proposed by de Jeu and Lewis \cite{LR}, so their conjecture gives another support for Conjecture \ref{main}.

Finally, we mention a variant of Conjecture \ref{main}.
It might be natural to expect that Conjecture \ref{main} holds for $\CH^2(X,1)_\ind$ of $\Z$-coefficients.
In relation to this strong version, we have the following result.
 
\begin{prop}[Proposition \ref{torprop}]\label{torpropmain}
For a $K3$ surface $X$, $\Aut_s(X)$ acts trivially on the torsion part $(\CH^2(X,1)_\ind)_\mathrm{tor}$.
\end{prop}

This is a direct consequence of the isomorphism between the torsion part of $\CH^2(X,1)_\ind$ and that of the Brauer group, which was proved in \cite[Theorem 1]{Kah}.
Since the torsion part of the target in \eqref{transregintro} is isomorphic to the Brauer group of $X$, it is plausible that the map \eqref{transregintro} induces an isomorphism between torsion parts.
If so, the integral version of Conjecture \ref{main} follows from Conjecture \ref{main} (Proposition \ref{rot}).
However, to the best of the author's knowledge, it is not clear whether \eqref{transregintro} induces the isomorphism between the torsion parts, so we are not sure about the integral version of Conjecture \ref{main}.

\subsection{Acknowledgement}
The author is sincerely grateful to Shohei Ma for many valuable suggestions and discussions on the contents of this paper.
In particular, Proposition \ref{abeliansurface} was taught by him.
This work was supported by JSPS KAKENHI 21H00971.

\subsection{Convention}
In this paper, we use the word \textit{variety} for an integral separated scheme of finite type over a field $k$.
For a $K3$ surface $X$, $\Aut_s(X)$ denotes the group of symplectic automorphisms of $X$.

\section{Translations on elliptic fibrations}
Let $X$ be a $K3$ surface and $\pi\colon X\to S$ be an elliptic fibration\footnote{
See Section 2.2 for the definition of elliptic fibrations in this paper.
Note that we assume the existence of a section.
Furthermore, since $K3$ surface are minimal, $\pi$ is always relatively minimal.}.
The set of sections of $\pi$ is denoted by $\mathrm{MW}(\pi)$, and has an abelian group structure induced by the elliptic fibration.
For each $D\in \mathrm{MW}(\pi)$, the translation by $D$ induces an automorphism of $X$, thus we have an injective map
\begin{equation}\label{mwtoauto}
\mathrm{MW}(\pi) \hookrightarrow \Aut(X).
\end{equation}
By the explicit description of 2-forms on $X$ in \cite[Section 5.13]{SS}, the translation acts trivially on $H^{2,0}(X)$, so the image of \eqref{mwtoauto} is in $\Aut_s(X)$.
In particular, we can regard the torsion part $\mathrm{MW}(\pi)_\mathrm{tor}$ as the subgroup of $\Aut_s(X)$ by the embedding \eqref{mwtoauto}.

\begin{defn}
For a $K3$ surface $X$, let $\mathrm{MW}_\mathrm{tor}\subset \Aut_s(X)$ denote the subgroup generated by elements of $\mathrm{MW}(\pi)_\mathrm{tor}$ where $\pi$ runs over all possible elliptic fibration structure on $X$.
\end{defn}
In this section, we prove the following Theorem \ref{newmainthm}.

\begin{thm}\label{newmainthm}
The subgroup $\mathrm{MW}_{\mathrm{tor}}$ acts trivially on $\CH^2(X,1)_{\mathrm{ind}}\otimes \Q$.
In particular, for a $K3$ surface such that $\mathrm{MW}_\mathrm{tor}=\Aut_s(X)$, Conjecture \ref{main} holds.
\end{thm}

Before proceeding the proof of Theorem \ref{newmainthm}, we will deduce Conjecture \ref{main} for a Kummer surfaces of product type by Theorem \ref{newmainthm}

\begin{cor}\label{Kummer}
For non-isogenus generic elliptic curves $E,F$, let $X=\mathrm{Km}(E\times F)$ be the Kummer surface associated with 
the product $E\times F$.
Then Conjecture\ref{main} is true for $X$.
\end{cor}
\begin{proof}
By  \cite[Theorem 5.3 and Section 4.1]{KK}, $\Aut_s(X)$ is generated by $28$ symplectic involutions induced by translations of 2-torsion sections with respect to some elliptic fibration structure on $X$.
Thus we can apply the latter part of Theorem \ref{newmainthm}.
\end{proof}

\begin{rem}
Let $X$ be a $K3$ surface with the finite automorphism group.
Such $K3$ surfaces are classified by Nikulin \cite{Nikulin}, and their automorphism groups are determined by Kondo \cite{Kondo}.
In \cite{Kondo}, for most cases, generators of  $\Aut_s(X)$ are given by translations of elliptic fibrations.
In particular, except\footnote{In these cases, Kond\=o constructs generators of $\Aut_s(X)$ using Torelli theorem, so we do not know whether they come from $\mathrm{MW}_{\mathrm{tor}}$ or not.} when $\mathrm{NS}(X) = U\oplus E_8\oplus E_8, U\oplus A_1^{\oplus 8}, U(2) \oplus A_1^{\oplus 7}$, we have $\mathrm{MW}_{\mathrm{tor}} = \Aut_s(X)$ for a $K3$ surface $X$ with the finite automorphism group.
Therefore, Conjecture \ref{main} holds for such $K3$ surfaces.
\end{rem}

In the remaining part of this section, we will prove Theorem \ref{newmainthm}, which is a consequence of the more general result, Proposition \ref{mainthm}.
Since we use base change arguments, throughout the rest of this section, we consider elliptic surfaces which are not necessarily $K3$ surfaces.

In Section 2.1, we list some properties on higher Chow cycles we use in this section.
In Section 2.2, we prove basic results about the group $\CH^2(X,1)_{\mathrm{ind}}$ for an elliptic surface $\pi\colon X\to S$.
In Section 2.3, we define a subgroup $F(\pi)\subset \CH^2(X,1)_{\mathrm{ind}}$ consisting of cycles supported on fibers, and show that translations acts trivially on $F(\pi)$.
In Section 2.4, we prove that translation acts trivially on cycles supported on sections and fibers, modulo cycles in $F(\pi)$.
This is done by constructing explicit symbols in the Milnor $K_2$ group $K_2^{M}(\C(X))$.
In Section 2.5, we prove Proposition \ref{mainthm} and finishes the proof of Theorem \ref{newmainthm}.

\subsection{Preliminaries}

For an equi-dimensional scheme $X$ of finite type over a field $k$ and $p,q\in \Z_{\ge 0}$, let $\CH^p(X,q)$ be the higher Chow group defined by Bloch (\cite{bloch}).
An element of $\CH^p(X,q)$ is called a \textit{$(p,q)$-cycle}.

For a morphism $f\colon X\to Y$ between smooth varieties over a filed $k$, the pull-back map $f^*\colon \CH^{p}(X,q)\to \CH^p(Y,q)$ is defined and satisfies $(g\circ f)^* = f^*\circ g^*$ for $X\xrightarrow{f} Y\xrightarrow{g} Z$.
For a proper morphism $f\colon X\to Y$ between equi-dimensional scheme $X$ of finite type over a field, the push-forward map $f_*:\CH^{\dim X-d}(X,q)\to \CH^{\dim Y-d}(Y,q)$ is defined and satisfies $(g\circ f)_* = g_*\circ f_*$ for $X\xrightarrow{f} Y\xrightarrow{g} Z$.

If $X$ is smooth over $k$, we have the intersection product map
\begin{equation*}
\CH^p(X,q)\times \CH^{p'}(X,q')\to \CH^{p+p'}(X,q+q'),
\end{equation*}
which is a bilinear map.
If $f\colon X\to Y$ is the morphism between smooth varieties, $f^*$ preserves the intersection product.

For a smooth projective variety $X$ over $\C$, we have isomorphisms $\CH^1(X)\simeq \mathrm{Pic}(X)$ and $\CH^1(X,1)\simeq \C^\times$.
Thus the intersection product induces the map 
\begin{equation}\label{intersectionproduct}
\mathrm{Pic}(X)\otimes\C^\times =\CH^1(X)\otimes_\Z \CH^1(X,1) \longrightarrow \mathrm{CH}^2(X,1).
\end{equation}
The image of this map is called the \textit{decomposable part} of $\mathrm{CH}^2(X,1)$ and is denoted by $\CH^2(X,1)_\dec$.
A \textit{decomposable cycle} is an element of $\CH^2(X,1)_\dec$.
The quotient $\CH^2(X,1)/\CH^2(X,1)_\dec$ is called the \textit{indecomposable part} of $\CH^2(X,1)$ and is denoted by $\mathrm{CH}^2(X,1)_{\mathrm{ind}}$. 
For a $(2,1)$-cycle $\xi$, $\xi_\ind$ denotes its image in $\mathrm{CH}^2(X,1)_{\mathrm{ind}}$.

Since $f^*$ preserves the intersection product, if $f\colon X\to Y$ is a morphism between smooth projective varieties over $\C$, the pull-back map $f^*\colon \CH^2(Y,1)\to \CH^2(X,1)$ induces the map 
\begin{equation*}
f^*\colon \CH^2(Y,1)_\ind \to \CH^2(X,1)_\ind.
\end{equation*}

For a surjective morphism $f\colon X\to Y$ between smooth projective varieties over $\C$ of same dimensions,
the following projection formula holds.
\begin{equation}\label{projectionformula}
f_*(\alpha\cdot (f^*\beta)) = (f_*\alpha)\cdot \beta\quad (\alpha\in \CH^p(X,q),\beta\in \CH^{p'}(X,q'))
\end{equation}
In particular, if we put $\alpha\in \CH^1(X)$ and $\beta\in \C^\times =\CH^1(Y,1)$ in \eqref{projectionformula}, we have $f_*(\CH^2(X,1)_{\dec}) \subset \CH^2(Y,1)_{\dec}$.
Thus, $f_*\colon \CH^2(X,1)\to\CH^2(Y,1)$ induces
\begin{equation}\label{indpushforward}
f_*\colon \CH^2(X,1)_\ind\to\CH^2(Y,1)_\ind.
\end{equation}

Moreover, if we set the degree of $f$ by $d=[\C(X):\C(Y)]$, by putting $\alpha = [X] \in \CH^0(X)$ in \eqref{projectionformula}, we have 
\begin{equation}\label{multd}
f_*f^*\beta = d\cdot \beta\quad (\beta\in \CH^p(X,q)).
\end{equation}

In this paper, we identify higher Chow groups as a homology group of Gersten complexes.
We use the following two cases.
For the proof for $(p,q) =(2,1)$, see, e.g., \cite{MS2} Corollary 5.3.
The case $(p,q)=(1,1)$ can be proved similarly.
In the following, for a equi-dimensional schemes $X$ of finite type over a field, and $r\in \Z_{\ge 0}$, $X^{(r)}$ denotes the set of all irreducible closed subsets of $X$ of codimension $r$.

For a smooth variety over a field $k$, the higher Chow group $\CH^2(X,1)$ is isomorphic to the homology group of the following complex.
\begin{equation*}
K_2^{\mathrm{M}} (k(X)) \xrightarrow{T} \displaystyle\bigoplus_{C\in X^{(1)}}k(C)^\times \xrightarrow{\mathrm{div}}\displaystyle\bigoplus_{p\in X^{(2)}}\Z\cdot p 
\end{equation*}
where the map $T$ denotes the tame symbol map from the Milnor $K_2$-group of the function field $k(X)$.
If $\dim \:X$ is less than 2, we regard the last term as $0$.
Note that $k(C)^\times$ coincides with the residue field of the generic point of $C$.

By this description,  each $(2,1)$-cycle is represented by a formal sum
\begin{equation}\label{formalsum}
\sum_j (C_j, \varphi_j)\in \displaystyle\bigoplus_{C\in X^{(1)}}k(C)^\times
\end{equation}
where $C_j$ are prime divisors on $X$ and $f_j\in k(C_j)^\times$ are non-zero rational functions on them such that $\sum_j {\mathrm{div}}_{C_j}(f_j) = 0$ as codimension 2 cycles on $X$.

Using the expression \eqref{formalsum}, the tame symbol map is given by the following formula.
\begin{equation*}
T(\{\varphi,\psi\}) = \sum_{C}(-1)^{\mathrm{ord}_C(\varphi)\mathrm{ord}_C(\psi)}\left(C,\varphi^{\mathrm{ord}_C(\psi)}\psi^{-\mathrm{ord}_C(\varphi)}\Big|_C\right)\quad (\varphi,\psi \in k(X)^\times)
\end{equation*}
where $\mathrm{ord}_C\colon k(X)^\times \to \Z$ denotes the order function along $C$.

For a smooth projective variety $X$ over $\C$, let $C$ be a prime divisor on $X$, $\alpha\in \C^\times = \CH^1(X,1)$ and $[C]\in \mathrm{Pic}(X)=\CH^1(X)$ be the class corresponding to $C$. 
Then, the intersection product $[C]\cdot \alpha\in \CH^2(X,1)$ is represented by $(C,\alpha)$ in the presentation (\ref{formalsum}). 

For an equi-dimensional scheme $X$ of finite type over $\C$,  the higher Chow group $\CH^1(X,1)$ is isomorphic to the kernel of the following map.
\begin{equation*}
\displaystyle\bigoplus_{C\in X^{(0)}}\C(C)^\times \xrightarrow{\mathrm{div}}\displaystyle\bigoplus_{p\in X^{(1)}}\Z\cdot p
\end{equation*}
We have the similar expression as \eqref{formalsum} for cycles in $\CH^1(X,1)$.

\subsection{Higher Chow cycles and elliptic fibration}
Hereafter we consider elliptic surfaces.
We use the following notatations.
\begin{enumerate}
\item $\pi\colon X\to S$ is a surjective morphism with connected fibers.
\item $X$ and $S$ are smooth projective varieties over $\C$ of dimension $2$ and $1$, respectively.
\item $z\colon S\to X$ is a section of $\pi$.
\item For a general closed point $s\in S$, the fiber $X_s=\pi^{-1}(s)$ is an elliptic curve with a unit $z(s)\in X_s$.
\end{enumerate}
Furthermore, we sometimes assume the following condition.
\begin{enumerate}
\setcounter{enumi}{4}
\item Each fiber $X_s$ does not contain $(-1)$-curves.
\end{enumerate}
If the condition (5) holds, $\pi$ is called \textit{relatively minimal}.

For a closed curve $C\subset X$, $C$ is called \textit{vertical} if $\pi(C)$ is a point, and \textit{horizontal} if $\pi(C)=S$.
Furthermore, if the restriction $\pi|_C\colon C\to S$ is isomorphism, $C$ is called a \textit{section}.
The image $z(S)$ of the section $z\colon S\to X$ is called the \textit{zero section}, and denoted by $Z$.
For an element in $\widetilde{\xi} \in \bigoplus_{C\in X^{(1)}}\C(C)^\times$, we have the canonical decomposition $\widetilde{\xi}=\widetilde{\xi}_h + \widetilde{\xi}_v$ such that $\widetilde{\xi}_h$ (resp. $\widetilde{\xi}_v$) is supported on horizontal (resp. vertical) curves. 

Let $\eta$ be the generic point of $S$.
Then $X_\eta =\pi^{-1}(\eta)$ is an elliptic curve over $\C(S)=\kappa(\eta)$ with the unit $Z_\eta$.
The following $1:1$ correspondence is crucial.
\begin{equation}\label{11corresp}
\{\text{horizontal curves on $X$}\} \longleftrightarrow \{\text{codimension 1 points on $X_\eta$}\}
\end{equation}
where the correspondence from left to right is given by $C\mapsto C_\eta$, and the inverse is given by taking the closure.
If a horizontal curve $C$ on $X$ corresponds to a codimension 1 point $p$ on $X_\eta$ by \eqref{11corresp}, 
the rational function field $\C(C)$ is canonically isomorphic to the residue field $\kappa(p)$.
Furthermore, the above correspondence induces a bijection
\begin{equation*}
\{\text{sections on $X$}\} \longleftrightarrow \{\text{$\C(S)$-rational points on $X_\eta$}\}
\end{equation*}
between subsets.

We have the commutative diagram
\begin{equation}\label{bigdiagram}
\begin{tikzcd}
0 \arrow[r]& 0 \arrow[r]\arrow[d] & K_2^M(\C(X)) \arrow[r,equal] \arrow[d,"T"]&  K_2^M(\C(X))  \arrow[r]\arrow[d,"T"] & 0 \\
0\arrow[r] & \displaystyle\bigoplus_{s\in S}\bigoplus_{C\in X_s^{(0)}}\C(C)^\times \arrow[r,"(*)"] \arrow[d,"\div"]&  \displaystyle\bigoplus_{C\in X^{(1)}}\C(C)^\times\arrow[r,"(**)"] \arrow[d,"\div"]&  \displaystyle\bigoplus_{p\in X_\eta^{(1)}}\kappa(p)^\times \arrow[r]\arrow[d] & 0\\
0\arrow[r] &  \displaystyle\bigoplus_{s\in S}\bigoplus_{p\in X_s^{(1)}}\Z\cdot p \arrow[r,equal] & \displaystyle\bigoplus_{p\in X^{(2)}}\Z \cdot p \arrow[r] & 0 \arrow[r] & 0
\end{tikzcd}
\end{equation}
where the vertical columns are Gersten complexes.
The map $(*)$ is a natural inclusion by regarding irreducible components of $X_s$ as prime divisors on $X$, and the map $(**)$ is a natural projection induced by the $1:1$ correspondence \eqref{11corresp}.
In particular, the horizontal rows are exact sequences.
This diagram induces the  following exact sequence.
\begin{equation}\label{localexact}
\bigoplus_{s\in S^{(1)}}\CH^1(X_s,1) \xrightarrow{i_*} \CH^2(X,1) \xrightarrow{j^*} \CH^2(X_\eta,1)
\end{equation}
This exact sequence coincides with the one induced by a localization sequence of higher Chow groups.
By a diagram chasing in \eqref{bigdiagram}, we can prove the following lemma.

\begin{lem}\label{verticalsectionlemma}
Let $\xi$ be a $(2,1)$-cycle on $X$.
\begin{enumerate}
\item If $\xi$ is represented by a cycle $\widetilde{\xi}\in \bigoplus_{C\in X^{(1)}}\C(C)^\times$ such that $\widetilde{\xi}_h=0$, then $\xi$ is in the image of $i_*$ in \eqref{localexact}.
\item Suppose that $j^*(\xi)\in \CH^2(X_\eta,1)$ is represented by a cycle in $\bigoplus_{p\in X_\eta^{(1)}}\kappa(p)^\times$ supported on $\C(S)$-rational points.
Then, $\xi$ is represented by a cycle $\widetilde{\xi}\in \bigoplus_{C\in X^{(1)}}\C(C)^\times$ such that the support of $\widetilde{\xi}_h$ is contained in sections of $X$.
\end{enumerate}
When the conclusion in $(2)$ holds, $\xi$ is called a section type.
\end{lem}

For a section $D$ of $\pi\colon X\to S$, the translation by $D_\eta$ induces the isomorphism $X_\eta\to X_\eta$ on the elliptic curve.
Consider the following condition.
\begin{equation*}
(\star)\quad \text{There exists a $\rho_D\in \Aut(X)$ such that $(\rho_D)_\eta$ is the translation by $D_\eta$.}
\end{equation*}
Note that $\rho_D$ is unique if it exists.
If $\pi$ is relatively minimal, $X$ is the Kodaira-N\'eron model of $X_\eta$, so $(\star)$ holds for any section.

\subsection{Cycles supported on fibers}

\begin{defn}
We define the subgroup $F(\pi)$ of $\CH^2(X,1)_{\ind}$ by the image of 
\begin{equation*}
\bigoplus_{s\in S^{(1)}}\CH^1(X_s,1) \xrightarrow{i_*} \CH^2(X,1) \to \CH^2(X,1)_{\ind}.
\end{equation*}
If $\pi\colon X\to S$ and $\pi \colon X'\to S'$ be elliptic fibrations, and $f\colon X'\to X, g\colon S' \to S$ are morphisms such that the diagram 
\begin{equation}\label{comm} 
\begin{tikzcd}
X\arrow[d,"\pi"'] & \arrow[l,"f"'] X' \arrow[d,"\pi'"] \\
S & \arrow[l,"g"] S'
\end{tikzcd}
\end{equation}
commutes, then we have $f_*(F(\pi')) \subset F(\pi)$ under the push-forward map in \eqref{indpushforward}.
\end{defn}

For a point $s\in S^{(1)}$, we have a subgroup
\begin{equation}\label{fiberdecomp}
\bigoplus_{C\in X_s^{(0)}}\C^\times \subset \mathrm{Ker}\left(\displaystyle\bigoplus_{C\in X_s^{(0)}}\C(C)^\times \xrightarrow{\mathrm{div}}\displaystyle\bigoplus_{p\in X_s^{(1)}}\Z\cdot p\right) = \CH^1(X_s,1).
\end{equation}
Let $I_s$ be the quotient of $\CH^1(X_s,1)$ by this subgroup.
Since the image of this subgroup by $\CH^1(X_s,1)\to \CH^2(X,1)$ is contained in the decomposable part, we have the surjective map
\begin{equation*}
\bigoplus_{s\in S^{(1)}}I_s \twoheadrightarrow F(\pi).
\end{equation*}
We will describe the group $I_s$ when $\pi\colon X\to S$ is relatively minimal.
A singular fiber $X_s$ is called \textit{multiplicative type} if the following cases occur. 
(See Figure \ref{singularfiberfig}.)
\begin{enumerate}
\item[$(\mathrm{I}_1)$] The fiber $X_s$ is a rational curve with a node.
\item[$(\mathrm{I}_2)$] The irreducible component of the fiber $X_s$ is $\Theta_0$ and $\Theta_1$ which are both isomorphic to $\P^1$ and intersect transversally at 2 points.
\item[$(\mathrm{I}_m)$] $(m\ge 3)$ The irreducible component of the fiber $X_s$ is $\Theta_0,\Theta_1,\ldots, \Theta_{m-1}$ which are all isomorphic to $\P^1$. 
We have $\Theta_0\cdot \Theta_1 = \Theta_1\cdot \Theta_2=\cdots =\Theta_{m-2}\cdot \Theta_{m-1} = \Theta_{m-1}\cdot \Theta_0=1$ and otherwise $\Theta_i\cdot \Theta_j=0$.
\end{enumerate}

\begin{figure}
\includegraphics[width = 130mm]{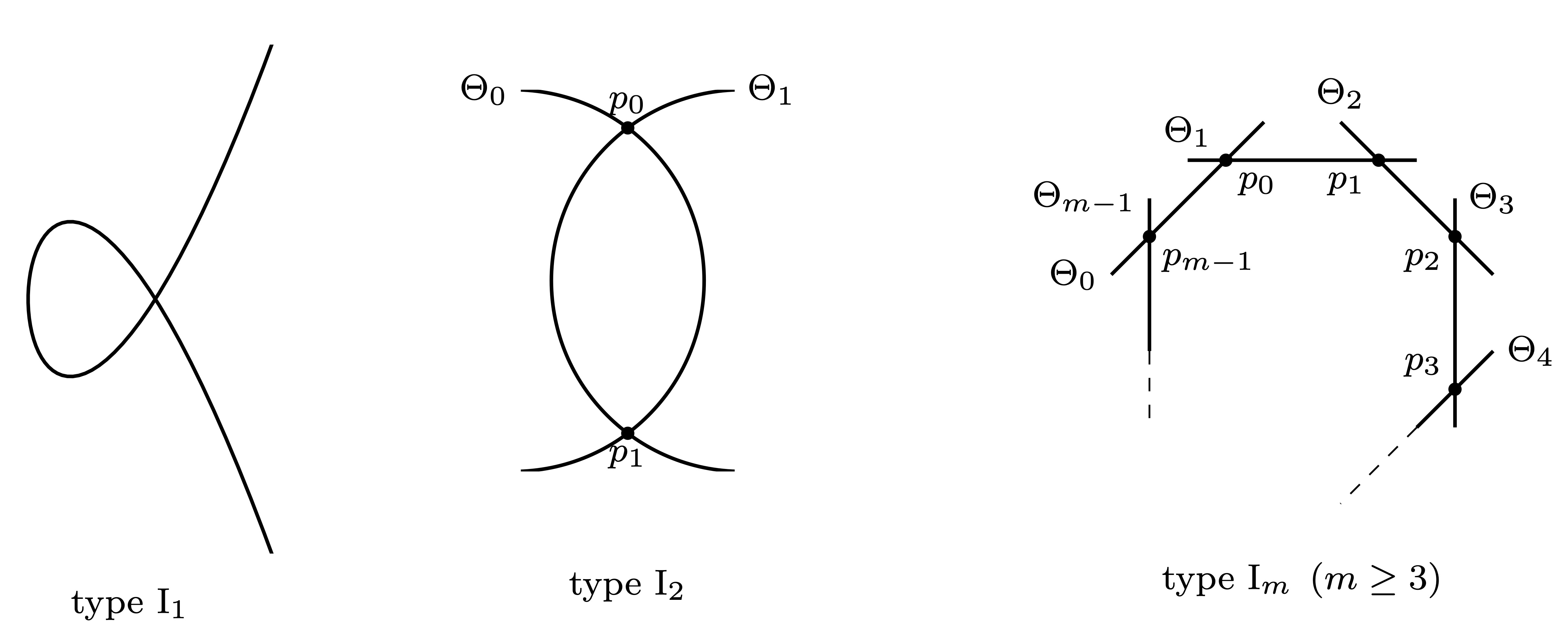}
\caption{The multiplicative singular fibers}
\label{singularfiberfig}
\end{figure}

Then we have the following.

\begin{prop}\label{generator}
Let $\pi\colon X\to S$ be a relatively minimal elliptic fibration.
For a closed point $s\in S$, we have
\begin{equation*}
I_s = \begin{cases}
\Z & (X_s\text{ is a multiplicative singular fiber.})\\
0 & (\text{otherwise})
\end{cases}
\end{equation*}
In particular, the rank of $F(\pi)$ is bounded by the number of multiplicative singular fibers.
\end{prop}

\begin{proof}
Let $\Theta_0,\Theta_1,\dots, \Theta_{m-1}$ be irreducible components of the fiber $X_s$.
Then cycles in $\xi\in \CH^1(X_s,1)$ can be represented by
\begin{equation}\label{repn}
\xi = (\Theta_0, \varphi_0)+(\Theta_1,\varphi_1)+\cdots +(\Theta_{m-1},\varphi_{m-1})
\end{equation}
where $\varphi_i\in \C(\Theta_i) ^\times$ are rational functions satisfying $\sum_{i=0}^{m-1}\div_{\Theta_i}(\varphi_i)=0$.

First, assume that $X_s$ is not multiplicative singular fiber and $\xi \in \CH^1(X_s,1)$.
We may assume $\xi\in\CH^1(X,1)$ is represented as in \eqref{repn}. 

If $X_s$ is smooth fiber, we have $\CH^1(X_s,1)=\C^\times$, so $I_s=0$.

If $X_s$ is of type II in the classification by Kodaira \cite[Theorem 6.2]{Kod}, $X_s$ is a rational curve with a cusp.
Let $\widetilde{X}_s\to X_s$ be the normalization and $p\in \widetilde{X}_s$ be a point on the cusp.
If $\xi=(X_s,\varphi)$ satisfies $\div_{X_s}(\varphi)=0$, we have $\div_{\widetilde{X_s}}(\varphi)=0$.
Then $\varphi\in \C^\times$, so $\xi=0$ in $I_s$. 

If $X_s$ is other type of additive singular fibers, all irreducible components are isomorphic to $\P^1$, and there exists an irreducible component $\Theta_{i_0}$ which intersects with the other components only at a single point $p_{0}$.
Since $\sum_{i}\div_{\Theta_i}(\varphi_i)=0$, the support of $\div_{\Theta_{i_0}}(\varphi_{i_0})$ is contained in $\{p_0\}$.
Then we have $\div_{\Theta_{i_0}}(\varphi_{i_0})=0$, and this implies $\varphi_{i_0}$ is constant, i.e., $\varphi_{i_0}\in \C^\times$.
Next, we can find an irreducible component $\Theta_{i_1}$ which intersects with the other components except $\Theta_{i_0}$ only at a single point $p_1$.
Since  $\sum_{i\neq i_0}\div_{\Theta_i}(\varphi_i)=0$, the support of $\div_{\Theta_{i_1}}(\varphi_{i_1})$ is contained in $\{p_1\}$.
Then we have $\div_{\Theta_{i_1}}(\varphi_{i_1})=0$, and this implies $\varphi_{i_1}$ is constant.
Continuing the same arguments, we can show that all rational functions appearing in \eqref{repn} is constant.
Thus $\xi$ is in the subgroup of  \eqref{fiberdecomp}. 
This implies $I_s=0$.

Secondly, assume that $X_s$ is of type $\mathrm{I}_1$.
Let $\widetilde{X}_s\to X_s$ be the normalization and $p_0,p_\infty\in \widetilde{X}_s$ be the points above the node.
Since $\widetilde{X}_s\simeq \P^1$, we can find a rational function $\psi\in \C(\widetilde{X}_s)^\times(=\C(X_s)^\times)$ such that $\div_{\widetilde{X}_s}(\psi) = p_0-p_\infty$.
Note that $\psi$ is determined up to constant multiplication by this relation.
Then $\psi$ satisfies $\div_{X_s}(\psi)=0$, so $(X_s,\psi)$ defines a nonzero element $\xi_1\in I_s$.
Clearly, it is non-torsion.
Let $\xi = (X_s,\varphi)$ be another $(1,1)$-cycle on $X_s$. 
Since $\div_{X_s}(\varphi) = 0$ on $X_s$,  the suppor of $\div_{\widetilde{X}_s}(\varphi)$ is contained in $\{p_0,p_\infty\}$, so we have $\div_{\widetilde{X}_s}(\varphi)= n\cdot p_0-n\cdot p_\infty$ for some $n\in \Z$.
This implies $\div_{\widetilde{X}_s}(\psi^{n})=\div_{\widetilde{X}_s}(\varphi)$, so $\varphi$ equals $\psi^n$ times a contant.
Thus we have $\xi=n\cdot \xi_1$ in $I_s$ and $I_s = \Z\cdot \xi_1$.

Finally, assume that $X_s$ is of type $\mathrm{I}_m\:\:(m\ge 2)$.
Hereafter we consider all index $i$ as elements of $\Z/m\Z$, e.g., we identify $i=0$ and $i=m$.
For $m=2$, let $p_0$ and $p_1$ be the intersection points of $\Theta_0$ and $\Theta_1$.
For $m\ge 3$, we label the intersection points $p_0,p_1,\ldots, p_{m-1}$ of irreducible components so that $\Theta_i\cap \Theta_{i+1} = \{p_i\}$.
For $i\in \Z/m\Z$, let $\psi_i$ be a rational function on $\Theta_i\simeq \P^1$ such that $\div_{\Theta_i}(\psi_i) = p_i-p_{i-1}$.
Such a rational function is uniquely determined up to constant multiplication and satisfies $\sum_{i}\div_{\Theta_i}(\psi_i)=0$.
Thus, these rational functions define a non-torsion element $\xi_m\in I_s$.
Let $\xi\in \CH^1(X,1)$ be another $(1,1)$-cycle, represented as in \eqref{repn}.
Since we have $\sum_{i}\div_{\Theta_i}(\varphi_i)=0$, the support of $\div_{\Theta_i}(\varphi_i)$ is contained in $\{p_i,p_{i-1}\}$.
Then there exists $n_i\in \Z$ such that $\div_{\Theta_i}(\varphi_i)=n_i\cdot \div_{\Theta_i}(\psi_i)$, so $\varphi_i$ is $\psi^{n_i}$ times a constant.
Since we have 
\begin{equation*}
0=\sum_{i=0}^{m-1}\div_{\Theta_i}(\varphi_i) = \sum_{i=0}^{m-1}n_i\cdot (p_i-p_{i-1}) = \sum_{i=0}^{m-1}(n_i-n_{i+1})p_i,
\end{equation*}
this implies $n_0=n_1=\cdots =n_{m-1}$.
Thus, if we put $n=n_0=n_1=\cdots =n_{m-1}$, we have $\xi = n\cdot \xi_m$ in $I_s$.
Thus we have $I_s = \Z\cdot \xi_m$.
\end{proof}

Assume that $\pi$ is relatively minimal.
For a section $D$, the translation by $D_\eta$ on $X_\eta$ always induces $\rho_D\in\Aut(X)$ since the condition $(\star)$ in the end of Section 2.2 is always satisfied.
Then we have the following.

\begin{prop}\label{singularfiberprop}
The pull-back map $\rho_D^*\colon \CH^2(X,1)_\ind \to \CH^2(X,1)_\ind$ preserves the subgroup $F(\pi)$.
Furthermore, $\rho_D^*\colon F(\pi)\to F(\pi)$ is the identity map.
\end{prop}

\begin{proof}
Since $\rho_D$ preserves fibers and we have $\rho_D^*=(\rho_D^{-1})_*$, the former part is clear.
For the latter part, it is enough to prove that $\rho_D^*\colon I_s\to I_s$ is an identity map.
First, we consider the case $X_s$ is $\mathrm{I}_m$-type fiber for $m\ge 3$.
Since $X$ is the Kodaira-N\'eron model for the elliptic curve $X_\eta$, so the group law on $X_\eta$ induces the group structure on the smooth locus $X_s^\sharp$ of $X_s$ (cf. \cite[Theorem 5.22]{SS}).
By considering the quotient by the identity component, we have a surjective morphism
\begin{equation*}
\varpi\colon X_s^\sharp \twoheadrightarrow \Z/m\Z
\end{equation*}
between group varieties \cite{neron}.
We label the irreducible components $\Theta_0,\Theta_1,\ldots, \Theta_{m-1}$ of $X_s$ so that $\varpi^{-1}(i)$ is contained in $\Theta_i$.
Then these components satisfies $\Theta_0\cdot \Theta_1 = \Theta_1\cdot \Theta_2=\cdots =\Theta_{m-2}\cdot \Theta_{m-1} = \Theta_{m-1}\cdot \Theta_0=1$, so they are labeled cyclically.
We label their intersection points $p_i = \Theta_i\cap \Theta_{i+1}$.
By Proposition \ref{generator}, if we take a rational function $\psi_i\in \C(\Theta_i)^\times$ such that $\div_{\Theta_i}(\psi_i)=p_i-p_{i-1}$, the $(1,1)$-cycle $\xi_m\sum_{i}(\Theta_i,\psi_i)$ is a generator of $I_s$.

Since $D$ is a section, the intersection $X_s\cap D$ is contained in the smooth locus $X_s^\sharp$.
Let $k=\varpi(X_s\cap D)\in \Z/m\Z$.
Then we have $\rho_D(\Theta_i) = \Theta_{i+k}$ and
\begin{equation*}
\rho_D(p_i) = \rho_D(\Theta_i\cap \Theta_{i+1}) = \rho_D(\Theta_i)\cap \rho_D(\Theta_{i+1}) = \Theta_{i+k}\cap \Theta_{i+k+1} = p_{i+k}.
\end{equation*}
for $i\in \Z/m\Z$.
This implies 
\begin{equation*}
\div_{\Theta_{i}}((\rho_D)^\sharp(\varphi_{i+k})) = (\rho_D)^{-1}(\div_{\Theta_{i+k}}(\varphi_{i+k})) = (\rho_D)^{-1}(p_{i+k}-p_{i+k-1}) = p_i-p_{i-1}.
\end{equation*}
So $(\rho_D)^\sharp(\varphi_{i+k})$ coincides with $\varphi_i$ coincide up to constant multiplication.
Thus we have 
\begin{equation*}
\rho_D^*(\xi_m) = \sum_i \left((\rho_D)^{-1}(\Theta_i), (\rho_D)^\sharp(\varphi_{i})\right) =\sum_i (\Theta_{i-k}, \varphi_{i-k}) = \xi_m  \quad \text{in }I_s.
\end{equation*}
Thus we have the result.

\begin{figure}
\includegraphics[width = 130mm]{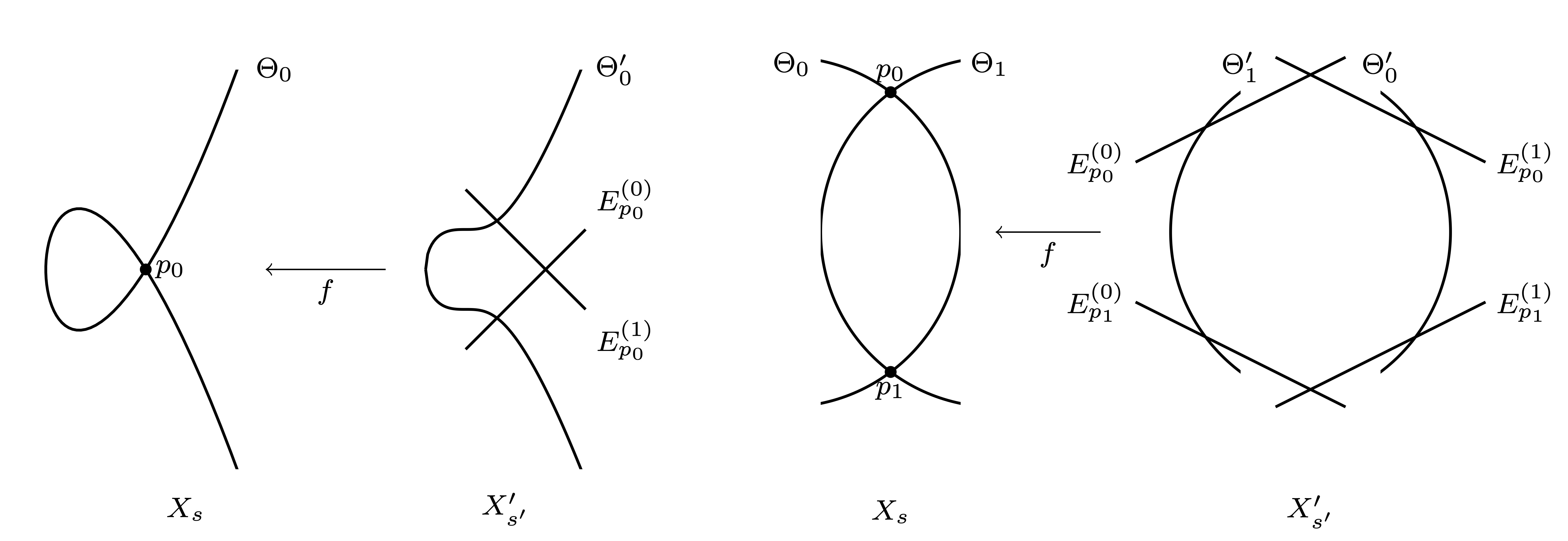}
\caption{From $I_m$-type to $I_{3m}$-type}
\label{basechangefig}
\end{figure}

Finally, we consider the case $X_s$ is $\mathrm{I}_1$ or $\mathrm{I}_2$ type singular fiber.
There exists a triple covering $S'\to S$ from a smooth curve $S'$ totally ramified at $s\in S$.
Let $s'\in S'$ be the point above $s\in S$ and $X'\to X\times_{S} S'$ be the resolution of singularities.
Then $\pi'\colon X'\to  X\times_{S} S' \to S'$ is an elliptic fibration.
The nodes $p_\bullet$ on $X_s$ become $A_2$-type singular points after the base change $X\times_{S} S' \to X$.
By the resolution of singularities $X'\to X\times_{S} S'$, we have two exceptional curves $E_{p_\bullet}^{(0)}$ and $E_{p_\bullet}^{(1)}$ over each $A_2$-type singular point.
Thus $X_{s'}$ becomes a singular fiber of type $I_{3m}$ (see figure \ref{basechangefig}). 
Furthermore, the base change of $\rho_D\colon X\to X$ induces the automorphism $\rho'_{D}\colon X'_{s'}\to X'_{s'}$ such that the left diagram in the \eqref{fibercommutative} commutes.
\begin{equation}\label{fibercommutative}
\begin{tikzcd}
X'_{s'} \arrow[r,"\rho'_D"]\arrow[d,"f"] & X'_{s'}\arrow[d,"f"]& & \CH^1(X'_{s'},1) \arrow[d,"f_*"] \arrow[r,"(\rho_D')^*"] & \CH^1(X'_{s'},1)\arrow[d,"f_*"]\\
X_s \arrow[r,"\rho_D"] & X_s && \CH^1(X_s,1) \arrow[r,"\rho_D^*"] &  \CH^1(X_s,1)
\end{tikzcd}
\end{equation}
Then the right diagram in \eqref{fibercommutative} commutes.
Note that $X'$ is not necessarily relatively minimal, but since $X'_{s'}$ does not contain $(-1)$-curves, by considering the model of $X'$ blowing down $(-1)$-curves, we have $\rho_D^*(\xi_{3m}) = \xi_{3m}$.
By the explicit description of the push-forward map $f_*\colon \CH^1(X'_{s'},1)\to \CH^1(X'_{s'},1)$, we have $f_*(\xi_{3m}) = \pm \xi_{m}$.
Then by the commutative diagram, we have $\rho_D^*(\xi_m) = \xi_m$.
\end{proof}

\subsection{Cycles of section type}
Let $\pi\colon X\to S$ be (not necessarily relatively minimal) elliptic fibration and $D\subset X$ be a section satisfying the condition $(\star)$.

\begin{prop}\label{ofsectiontype}
Let $\xi \in \CH^2(X,1)$ be a section type in the sense of Lemma \ref{verticalsectionlemma}.
Then we have
\begin{equation*}
\rho_D^*(\xi_\ind) -\xi_\ind \in F(\pi).
\end{equation*}
\end{prop}

\begin{proof}
By the assumption, $\xi$ is represented by a $\widetilde{\xi}\in \mathrm{Ker}\left(\bigoplus_{C\in X^{(1)}}\C(C)^\times\xrightarrow{\div}\bigoplus_{p\in X^{(2)}}\Z\cdot p\right)$ such that
\begin{equation*}
\widetilde{\xi}_h = (C_1,\varphi_1)+(C_2,\varphi_2)+\cdots +(C_n,\varphi_n)
\end{equation*}
where $C_1,C_2,\ldots, C_n$ are sections.
For each $i=1,2,\ldots, n$, let $\widetilde{\varphi}_i\in \C(X)^\times$ be the rational function defined by the composition
\begin{equation*}
\begin{tikzcd}
X\arrow[r,"\pi"]& S & \arrow[l,"\pi|_{C_i}"',"\sim"] C_i \arrow[r,"\varphi_i"] &\P^1.
\end{tikzcd}
\end{equation*}
Furthermore, since the codimension 1-cycles $C_{i,\eta}+Z_\eta$ and $\rho_D^{-1}(C_{i})_\eta+D_\eta$ on the elliptic curve $X_\eta$  are rationally equivalent, there exists a $\psi_i\in \C(X_\eta)^\times$ such that 
\begin{equation*}
\div_{X_\eta}(\psi_i) = C_{i,\eta}+Z_\eta - \rho_D^{-1}(C_{i})_\eta- D_\eta \quad \text{in }\bigoplus_{p\in X_\eta^{(1)}}\Z\cdot p.
\end{equation*}
By the identification $\C(X_\eta) = \C(X)$, this implies that 
\begin{equation*}
\div_{X}(\psi) = C_i + Z - \rho_D^{-1}(C_{i}) - D +(\text{vertical curves}).
\end{equation*}
Since the support of $\div_X(\widetilde{\varphi}_i)$ is contained in vertical curves, by the explicit description of the tame symbol map, we have
\begin{equation*}
T(\{\widetilde{\varphi}_i,\psi_i\}) = (C_i,\widetilde{\varphi}_i|_{C_i}) + (Z, \widetilde{\varphi}_i|_Z) - (\rho_D^{-1}(C_{i}),\widetilde{\varphi}_i|_{\rho_D^{-1}(C_i)}) - (D,\widetilde{\varphi}_i|_{D}) + (\text{vertical cycles})
\end{equation*}
where vertical cycles means a cycle whose support is contained in vertical curves.

Put $\Xi = \{\widetilde{\varphi}_1,\psi_1\}+\{\widetilde{\varphi}_2,\psi_2\}+\cdots +\{\widetilde{\varphi}_n\psi_n\} \in K_2^M(\C(X))$, then we have
\begin{equation*}
\begin{aligned}
T(\Xi) &= \sum_i (C_i,\widetilde{\varphi}_i|_{C_i}) + \left(Z,\prod_i\widetilde{\varphi}_i|_Z\right)\\
&- \sum_{i}(\rho_D^{-1}(C_{i}),\widetilde{\varphi}_i|_{\rho_D^{-1}(C_i)}) - \left(D,\prod_i\widetilde{\varphi}_i|_D\right) + (\text{vertical cycles}).
\end{aligned}
\end{equation*}
By definition, we have $\widetilde{\varphi}_i|_{C_i} = \varphi_i$ and $\widetilde{\varphi}_i|_{\rho_D^{-1}(C_i)} = \rho_D^\sharp(\varphi_i)$, so we have 
\begin{equation}\label{chuukan}
T(\Xi) = \widetilde{\xi}  - \rho_D^*\left(\widetilde{\xi}\right) + \left(Z,\prod_i\widetilde{\varphi}_i|_Z\right) -  \left(D,\prod_i\widetilde{\varphi}_i|_D\right) +\widetilde{\xi}'
\end{equation}
where $\widetilde{\xi}'\in \bigoplus_{C\in X^{(1)}}\C(C)^\times$ is a vertical cycle.
We will show the following claim.
\begin{claim}\label{claim}
$\displaystyle\prod_i\widetilde{\varphi}_i|_Z, \: \displaystyle\prod_i\widetilde{\varphi}_i|_D \in \C^\times$.
\end{claim}
\begin{proof}[Proof of Claim \ref{claim}]
Since $\pi|_Z\colon  Z\to S$ and $\pi|_D\colon D\to S$ are isomorphisms, it is enough to prove
\begin{equation}\label{ETS}
\pi_*\left(\div_{Z}\left(\displaystyle\prod_i\widetilde{\varphi}_i|_Z\right)\right)=0,\quad \pi_*\left(\div_{D}\left(\displaystyle\prod_i\widetilde{\varphi}_i|_D\right)\right)=0
\end{equation}
as a codimension 1-cycles on $S$.
Since proofs of the both cases are similar, we will prove only the first equality.
By definition, the left-hand side of \eqref{ETS} can be transformed into 
\begin{equation}\label{ETStrans}
\begin{aligned}
&\pi_*\left(\div_{Z}\left(\displaystyle\prod_i\widetilde{\varphi}_i|_Z\right)\right) = \pi_*\left(\sum_i\div_{Z}(\widetilde{\varphi}_i|_Z)\right)= \sum_i\pi_*\left(\div_Z(\widetilde{\varphi}_i|_Z)\right)  \\
&= \sum_i \div_{S}(\varphi_i\circ (\pi|_{C_i})^{-1})=\pi_*\left(\sum_{i} \div_{C_i}(\varphi_i)\right) = \pi_*(\div(\widetilde{\xi}_h)).\quad \cdots (*)
\end{aligned}
\end{equation}
Since $\widetilde{\xi}\in \mathrm{Ker}\left(\bigoplus_{C\in X^{(1)}}\C(C)^\times\xrightarrow{\div}\bigoplus_{p\in X^{(2)}}\Z\cdot p\right)$, we have $\div(\widetilde{\xi}_h)=-\div(\widetilde{\xi}_v)$.
For each component $(F_j,\theta_j)$ of $\widetilde{\xi}_v$, $F_j$ is vertical, so we have $\pi_*(\div_{F_j}(\theta_j)) = 0$.
This implies $\pi_*(\div(\xi_v))=0$, so $(*)$ in  \eqref{ETStrans} is 0, thus we have proved the first equation in \eqref{ETS}.
\end{proof}
By Claim \ref{claim}, the element $\widetilde{\xi}''=  \left(Z,\prod_i\widetilde{\varphi}_i|_Z\right) -  \left(D,\prod_i\widetilde{\varphi}_i|_D\right)$ represents a decomposable cycle $\xi''$.
By \eqref{chuukan}, we have
\begin{equation*}
T(\Xi) = \widetilde{\xi}  - \rho_D^*(\widetilde{\xi}) + \widetilde{\xi}' + \widetilde{\xi}''.
\end{equation*}
Since $\widetilde{\xi},\rho_D^*(\widetilde{\xi}),T(\Xi),\widetilde{\xi}''\in\mathrm{Ker}\left(\bigoplus_{C\in X^{(1)}}\C(C)^\times\xrightarrow{\div}\bigoplus_{p\in X^{(2)}}\Z\cdot p\right)$, $\xi'$ also lies in the kernel.
Thus $\widetilde{\xi}'$ represents a $(2,1)$-cycle $\xi'$.
Then by Lemma \ref{verticalsectionlemma}, $\xi' \in \mathrm{Im}(i_*)$.
By the equation above, we have $0=\xi-\rho_D^*(\xi)+\xi' + \xi''$ and since $\xi''$ is decomposable, we have
\begin{equation*}
\rho_D^*(\xi_\ind) -\xi_{\ind} = \xi_\ind' \in F(\pi).
\end{equation*}
This finishes the proof.
\end{proof}

\subsection{Proof of Theorem \ref{newmainthm}}
Finally, we can prove the following.

\begin{prop}\label{mainthm}
Let $\pi\colon X\to S$ be a relatively minimal elliptic fibration and $D\subset X$ be a section.
Assume that either of the following conditions holds.
\begin{enumerate}
\item[$(\mathrm{i})$] $D$ is a torsion section.
\item[$(\mathrm{ii})$] $\pi$ has no multiplicative singular fiber.
\end{enumerate}
Then, $\rho_D\colon \CH^2(X,1)_\ind\otimes \Q\to \CH^2(X,1)_\ind\otimes \Q$ is the identity map.
\end{prop}
\begin{proof}
First, we will prove that 
\begin{equation}\label{toshow}
\rho_D^*(\xi_\ind)-\xi_\ind \in F(\pi)\otimes \Q
\end{equation}
for any section $D$.

Let $\xi \in \CH^2(X,1)$ and we take a lift $\widetilde{\xi}\in \bigoplus_{C\in X^{(1)}}\C(C)^\times$.
Let denote $\widetilde{\xi}_h = (C_1,\varphi_1)+(C_2,\varphi_2)+\cdots (C_n,\varphi_n)$.
Since $C_i$ are horizontal, $\C(C_i)/\C(S)$ is a finite extension of fields.
We embed them in an algebraic closure $\overline{\C(S)}$, and take a finite Galois extension $K/\C(S)$ in $\overline{\C(S)}$ such that $K$ contains all $\C(C_1),\C(C_2),\ldots, \C(C_n)$.
Let $S'$ be a smooth projective curve whose function field is isomorphic to $K$, and $g\colon S'\to S$ be the finite morphism induced by $\C(S)\hookrightarrow K$.
Let $X'\to X\times_{S} S'$ be the resolution of singularities.
Then $\pi'\colon X'\to X\times_{S} S'\to S' $ is a (not necessarily relatively minimal) elliptic fibration and the morphism $f\colon X'\to X\times_{S} S'\to X$ and $g$ fits into the commutative diagram \eqref{comm}.

Let $\eta'$ be the generic point of $S'$.
We have the following commutative diagram.
\begin{equation*}
\begin{tikzcd}
\CH^2(X,1) \arrow[r,"j^*"] \arrow[d,"f^*"]& \CH^2(X_\eta, 1) \arrow[d,"f_{\eta}^*"] \\
\CH^2(X',1) \arrow[r,"(j')^*"] & \CH^2(X'_{\eta'},1) 
\end{tikzcd}
\end{equation*}
Thus we have 
\begin{equation}\label{pullback}
(j')^*(f^*(\xi)) = f_{\eta}^*(j^*(\xi)).
\end{equation}
Furthermore, since $K$ contains the field $\C(C_i)$, we have the decomposition
\begin{equation}\label{decomp}
C_{i,\eta}\times_\eta \eta' = \bigsqcup_{j=1}^{m_i}C_{i,\eta'}^{(j)}
\end{equation}
where $C_{i,\eta'}^{(j)}$ are $K=\C(S')$-rational points on $X'_{\eta'}$.
Then by the decomposition \eqref{decomp} and the explicit description of the flat pull-back map on Gersten complex \cite{Rost}, the right-hand side of \eqref{pullback} is represented by 
\begin{equation*}
\sum_{i=1}^n\sum_{j=1}^{m_i}(C_{i,\eta'}^{(j)}, \varphi_i) \in \displaystyle\bigoplus_{p\in (X'_{\eta'})^{(1)}}\kappa(p)^\times.
\end{equation*}
Thus, by Lemma \ref{verticalsectionlemma}, $f^*(\xi)\in \CH^2(X',1)$ is a section type.
The base change of $\rho_D\colon X\to X$ induces the automorphism $\rho_D'\colon X'\to X'$ such that the diagram
\begin{equation}\label{basechangecomm}
\begin{tikzcd}
X' \arrow[r,"\rho'_D"]\arrow[d,"f"] & X'\arrow[d,"f"]& \\
X \arrow[r,"\rho_D"] & X &&
\end{tikzcd}
\end{equation}
commute, then by Proposition \ref{ofsectiontype}, we have $(\rho_D')^*f^*(\xi_\ind) - f^*\xi_\ind\in F(\pi')$.
Since $f_*(F(\pi')) \subset F(\pi)$, we have
\begin{equation*}
f_*(\rho_D')^*f^*(\xi_\ind) - f_*f^*\xi_\ind\in F(\pi).
\end{equation*}
By the relation $(\rho_D')^* = (\rho_D')^{-1}_*$ and the commutative diagram \eqref{basechangecomm},
\begin{equation*}
f_*(\rho_D')^*f^*(\xi_\ind) = f_* (\rho_D')^{-1}_*f^*(\xi_\ind) = (\rho_D^{-1})_*f_*f^* = \rho_D^*f_*f^*(\xi_\ind).
\end{equation*}
Finally, by \eqref{multd}, we have $N(\rho_D^*(\xi_\ind)-\xi_\ind) \in F(\pi)$ where $N=[K:\C(S)]$, so this implies \eqref{toshow}.

If we assume the condition (ii), we have $F(\pi)=0$ by Proposition \ref{generator}, the statement immediately follows from \eqref{toshow}.

We assume the condition (i).
By \eqref{toshow}, there exists $\xi'\in F(\pi)\otimes \Q$ such that $\xi' =\rho_D^*(\xi_\ind)-\xi_\ind$.
Since we have $\rho_D^*(\xi') = \xi'$ by Proposition \ref{singularfiberprop}, for any $m\in \Z_{>0}$, we have
\begin{equation*}
(\rho_D^m)^*(\xi_\ind) = \xi_{\ind}  + m\xi'.
\end{equation*}
If we take $m$ as the order of $D_\eta$, the left-hand side equals $\xi_\ind$. 
Thus we have $m\xi'=0$, so $\xi'=0$ in $\CH^2(X,1)_\ind\otimes \Q$.
This implies $\rho_D^*(\xi_\ind) = \xi_\ind $ in $\CH^2(X,1)_\ind \otimes \Q$, so we have the result.
\end{proof}

Finally, we can prove Theorem \ref{newmainthm}
\begin{proof}[Theorem \ref{newmainthm}]
Let $X$ be a $K3$ surface and  $\pi\colon X\to S$ be an elliptic fibration.
Since $X$ is a minimal surface, $\pi$ is relatively minimal.
Thus, we can apply Proposition \ref{mainthm} for $\pi$, thus $\mathrm{MW}(\pi)_\mathrm{tor}$ acts trivially on $\CH^2(X,1)_\ind\otimes \Q$.
Thus $\mathrm{MW}_\mathrm{tor}$ acts trivially on $\CH^2(X,1)_\ind\otimes \Q$.
\end{proof}

For an abelian variety $A$, the following analogue of Proposition \ref{mainthm} holds
This result and proof below was taught by Ma Shohei.
\begin{prop}\label{abeliansurface}
For a torsion element $a\in A_{\mathrm{tor}}$ of an abelian variety $A$, let $T_a\colon A\to A$ denote the translation by $a$.
Then 
\begin{equation*}
T_a^*\colon \CH^p(X,q)\otimes \Q\to \CH^p(X,q)\otimes \Q
\end{equation*}
is the identity map.
\end{prop}

\begin{proof}
We use the Fourier-Mukai transform for $A$.
Let $\hat{A}$ denote the dual abelian variety,  $\mathcal{P}$ be the Poincar\'e line bundle on $A\times \hat{A}$, and $p_1,p_2$ be 1st and 2nd projection from $A\times \hat{A}$.
We denote $\bigoplus_{p}\CH^p(X,q)$ by $\CH(X,q)$.
Then the Fourier-Mukai transform $F_{\mathcal P}\colon \CH(A,j)\otimes \Q\to  \CH(\hat{A},j)\otimes \Q$ is defined by $F_{\mathcal P}(\xi) = (p_2)_*(\exp([\mathcal{P}])p_1^*(\xi))$ where $\exp([\mathcal{P}]) = 1+\frac{[\mathcal{P}]}{1!}+\frac{[\mathcal{P}]^2}{2!}+\cdots$.
Since $F_{\mathcal P}$ has inverse, this is an isomorphism.
Then for $a\in A$, the translation $T_a$ satisfies
\begin{equation*}
F_{\mathcal P}\circ T_a^*(\xi) = \exp([\mathcal{P}|_{\{a\}\times \hat{A}}])F_{\mathcal{P}}(\xi).
\end{equation*}
by \cite[Proposition 4(ii)]{beau}.
Since $A\to \mathrm{Pic}(\hat{A}); a\mapsto [\mathcal{P}|_{\{a\}\times \hat{A}}]$ is an isomorphism of groups and $a$ is a torsion, $\exp([\mathcal{P}|_{\{a\}\times \hat{A}}])=1$ in $\CH(\hat{A},q)\otimes\Q$.
This shows that $T_a^* = \id$ on $\CH(X,q)\otimes \Q$.
\end{proof}

\section{Conditional results}
In this section, we give a conditional results on Conjecture \ref{main}, assuming some general conjectures on motives.

\subsection{Consequence from the conservativity conjecture}
First, we briefly review basic results on mixed motives following \cite{ayoub}.

Let $\mathbf{Chow}(\C;\Q)$ be the category of Chow motives over $\C$ with coefficients in $\Q$ and $\mathbf{DM}_{\mathrm{gm}}(\C;\Q)$ be the Voevodsky's category of geometric motives over $\C$ with coefficients in $\Q$.
It is known that $\mathbf{DM}_{\mathrm{gm}}(\C;\Q)$ is pseudo-abelian, i.e., for each idempotent map, its kernel and cokernel exist.
We have a fully faithful embedding 
\begin{equation*}
\mathbf{Chow}(\C;\Q)\hookrightarrow \mathbf{DM}_{\mathrm{gm}}(\C;\Q)
\end{equation*}
For an object $M\in \mathbf{DM}_{\mathrm{gm}}(\C;\Q)$ and $p,q\in \Z_{\ge 0}$, the motivic cohomology is defined by\footnote{Note that $\mathbf{DM}_{\mathrm{gm}}(\C;\Q)$ is a full subcategory of the category $\mathbf{DM}(\C;\Q)$ of mixed motives over $\C$ with coefficients in $\Q$.} 
\begin{equation*}
H^p_{\mathcal{M}}(M,\Q(q)) = \mathrm{Hom}_{\mathbf{DM}_{\mathrm{gm}}(\C;\Q)}(M,\Q(q)[p]).
\end{equation*}
In particular, for a smooth variety $X$, the motivic cohomology $H^p_{\mathcal{M}}(X,\Q(q))$  is defined by $\mathrm{Hom}_{ \mathbf{DM}_{\mathrm{gm}}(\C;\Q)}(M(X),\Q(q)(p))$, where $M(X)$ be the motive associated to $X$.
Furthermore, we have the canonical isomorphism $\CH^p(X,2q-p)\otimes \Q \simeq H^p_{\mathcal{M}}(X,\Q(q))$ where $\CH^p(X,q)$ is the Bloch's higher Chow group.
In particular, in the case $(p,q)=(3,2)$, we have the isomorphism
\begin{equation}\label{Chowmot}
\CH^2(X,1)\otimes\Q \simeq H^3_{\mathcal{M}}(X,\Q(2)).
\end{equation}

Let $\mathbf{D}(\Q)$ be the derived category of $\Q$-vector spaces and $H_B:\mathbf{DM}_{\mathrm{gm}}(\C;\Q) \to \mathbf{D}(\Q)$ be the Betti realization, i.e., the functor between triangulated categories defined by sending $M(X)$ to the singular chain complex $S^\bullet (X^\an;\Q)$.
Ayoub proposed the following ``conservativity conjecutre".
\begin{conje}\label{conservative}
\cite[Conjecture 2.1]{ayoub}
The functor $H_B$ is conservative.
In other words, if a morphism $f\colon M\to N$ in $\mathbf{DM}_{\mathrm{gm}}(\C;\Q)$ satisfies that $H_B(f)$ is isomorphism, then $f$ itself is an isomorphism.
\end{conje}

Next, we recall crucial results on motives of surfaces.
For a smooth projective surface $X$ over $\C$, we denote the $\Q$-linear subspace of $H^2(X,\Q)$ generated by algebraic cycles by $\mathrm{NS}(X)_\Q$, and it orthogonal complement (with respect to the cup product) by $T(X)_\Q$.
We have the following ``refined Chow K\"unneth decomposition" in the category $\mathbf{Chow}(\C;\Q)$.
\begin{thm}
\cite[Proposition 7.2.3, Theorem 7.3.10]{KMP} For a smooth projective surface $X$ over $\C$, the Chow motive $h(X)$ admits the splitting
\begin{equation*}
\begin{aligned}
&h(X) = h^0(X)\oplus h^1(X)\oplus h^2(X) \oplus h^3(X) \oplus h^4(X)\\
&h^2(X) = h^2_{\mathrm{alg}}(X)\oplus t_2(X)
\end{aligned}
\end{equation*}
where the Betti-realization of $h^i(X)$ is $H^i(X^\an,\Q)$, and the Betti-realization of $h^2_{\mathrm{alg}}(X)$ and $t_2(X)$ are $\mathrm{NS}(X)_\Q$ and $T(X)_\Q$, respectively.
Furthermore, any isomorphism between smooth projective surfaces preserves the above decomposition.
\end{thm}
The component $t_2(X)$ is called \textit{transcendental part of Chow motives}.
As we mentioned in the introduction, the relation between $t_2(X)$ and $\CH^2(X,1)_\ind$ is the following.

\begin{thm}\cite[Theorem 2]{Kah}\label{kahisom}
There exists the following canonical isomorphism.
\begin{equation*}
\CH^2(X,1)_\ind \otimes \Q \simeq H^3_{\mathcal{M}}(t_2(X),\Q(2))
\end{equation*}
\end{thm}

By assuming the conservativity conjecture, we can show the following conditional results on Conjecture \ref{main}.

\begin{prop}\label{condproof}
Assume Conjecture \ref{conservative} holds.
Then, Conjecture \ref{main} holds for a $K3$ surface $X$ such that $\Aut_s(X)$ is generated by finite orders.
\end{prop}
\begin{proof}
It is enough to show that for any elements $\rho\in \Aut_s(X)$ of a finite order, $\rho$ acts trivially on $\CH^2(X,1)_\ind\otimes \Q$.
Using the isomorphism in Theorem \ref{kahisom} and by the definition of the motivic cohomology, it is enough to show that $\rho$ acts trivially on the Chow motive $t_2(X)$.

Let $m$ be the order of $\rho$.
Then the endomorphism $(\mathrm{id}+\rho^*+(\rho^2)^*+\cdots + (\rho^{m-1})^*)/m\colon t_2(X)\to t_2(X)$ is idempotent.
Thus the $\rho$-invariant part $t_2(X)^\rho$ of $t_2(X)$ is defined in $\mathbf{DM}_{\mathrm{gm}}(\C;\Q)$.
The Betti realization of the natural morphism $t_2(X)^\rho \to t_2(X)$ is $T(X)^{\rho^*}_\Q\hookrightarrow T(X)_\Q$.
Since $\rho$ is an symplectic automorphism, $\rho^*$ acts trivially on the transcendental lattice $T(X)$ (\cite[p.~ 330]{HuyK3}), $T(X)^{\rho^*}_\Q\hookrightarrow T(X)_\Q$ is an isomorphism.
Therefore, by Conjecture \ref{conservative}, $t_2(X)^\rho\to t_2(X)$ is isomorphism, i.e., $\rho$ acts trivially on $t_2(X)$.
\end{proof}

\subsection{Relation with injectivity of the Regulator map}
In this section, we explain that the injectivity of the regulator map implies Cojecture \ref{main}.
For a $K3$ surface $X$, the following \textit{regulator map} plays an important role in the study of $\CH^2(X,1)_\ind$.
\begin{equation}\label{regulator}
\CH^2(X,1) \to H^3_{\mathcal{D}}(X,\Z(2))= \dfrac{H^2(X,\C)}{F^2H^2(X,\C)+H^2(X,\Z(2))} = J(H^2(X,\Z))
\end{equation}
Here the target is a generalized complex torus, i.e., a quotient of a $\C$-vector space by a  non-saturated discrete lattice.
By the explicit formula on the regulator map, the restriction of \eqref{regulator} to the decomposable part is given by
\begin{equation}\label{decreg}
\CH^2(X,1)_\dec \to \mathrm{NS}(X)\otimes (\C/\Z(1)); \quad (C,\alpha)\mapsto [C]\otimes \log\alpha.
\end{equation}
Since $\mathrm{Pic}(X)=\mathrm{NS}(X)$ for $K3$ surfaces, \eqref{decreg} is isomorphism.
Let $T(X)^\vee = \mathrm{Hom}_\Z(T(X),\Z)$ be the dual lattice of the transcendental lattice $T(X)$.
We can regard $T(X)^\vee$ as a Hodge structure of weight 2.
By the unimodularity of $H^2(X,\Z)$, we have the map $H^2(X,\Z) \xrightarrow{\:\sim \:} H^2(X,\Z)^\vee \to T(X)^\vee$.
Since $T(X)$ and $\mathrm{NS}(X)$ are primitive lattices of $H^2(X,\Z)$ and orthogonal to each other, this map is surjective.
Thus, this morphism induces  the following exact sequence of Hodge structures.
\begin{equation}\label{morHS}
\begin{tikzcd}
0\arrow[r] & \mathrm{NS}(X) \arrow[r] & H^2(X,\Z) \arrow[r] &  T(X)^\vee \arrow[r] & 0.
\end{tikzcd}
\end{equation}
By \eqref{decreg} and \eqref{morHS}, the regulator map \eqref{regulator} induces the map
\begin{equation}\label{transreg}
r\colon \CH^2(X,1)_\ind \to \dfrac{T(X)^\vee_\C}{F^2T(X)^\vee_\C+T(X)^\vee} = J(T(X)^\vee).
\end{equation}
The map \eqref{transreg} is called the \text{transcendental regulator map}, and used for detecting indecomposable cycles.
Using the notations above, we prove the following.

\begin{prop}\label{BH}
Let $X$ be a $K3$ surface such that the map 
\begin{equation}\label{transregq}
r\otimes \id\colon\CH^2(X,1)_\ind\otimes \Q \to J(T(X)^\vee)\otimes \Q 
\end{equation}
is injective.
Then Conjecture \ref{main} holds for $X$.
\end{prop}
\begin{proof}
Since symplectic automorphisms act trivially on $T(X)$, they also act trivially on $J(T(X)^\vee)\otimes \Q$.
Thus, the injectivity of \eqref{transregq} implies that symplectic automorphisms act trivially on $\CH^2(X,1)_\ind \otimes \Q$.
\end{proof}

For a Zariski open subset $U$ of $X$, we have the cycle class map
\begin{equation*}
\CH^2(U,2)\otimes \Q\to \mathrm{Hom}_{\mathrm{MHS}}(\Q(0),H^2(U,\Q(2)))
\end{equation*}
where the target denotes the $\Q$-linear space of morphisms in the category of $\Q$-mixed Hodge structures.
By taking the direct limit by $U\subset X$, the cycle class map induces
\begin{equation}\label{cyclemap}
K_2^{M}(\C(X))\otimes \Q=\CH^2(\Spec \C(X),2)\otimes \Q \to \varinjlim_{U\subset X}\mathrm{Hom}_{\mathrm{MHS}}(\Q(0),H^2(U,\Q(2))).
\end{equation}
The surjectivity of \eqref{cyclemap} is a special case of the conjecture (S3) proposed in \cite{LR}, which is called \textit{amended Beilinson's Hodge conjecture} by them.
By \cite[Corollary 4.14]{LR}, the surjectivity of \eqref{cyclemap} is equivalent to the injectivity \eqref{transregq}.
Thus, we have the following.
\begin{cor}
If the amended Beilinson's Hodge conjecture holds for $X$, Conjecture \ref{main} holds for such $X$.
\end{cor}

\section{Results on the torsion part}
On symplectic actions on the torsion part of $\CH^2(X,1)_\ind$, we have the following.
\begin{prop}\label{torprop}
Let $X$ be a complex $K3$ surface and $\Aut_s(X)$ be the symplectic automorphism group.
Then $\Aut_s(X)$ acts trivially on $(\CH^2(X,1)_\ind)_\mathrm{tor}$.
\end{prop}
\begin{proof}
By \cite[Theorem 1]{Kah}, we have the following isomorphism 
\begin{equation}\label{Brauer}
\mathrm{Br}(X)(1) \xrightarrow{\:\sim\:} (\CH^2(X,1)_\ind)_{\mathrm{tor}} 
\end{equation}
where $\mathrm{Br}(X)(1) = \varinjlim_{n}{}_n\mathrm{Br}(X)\otimes \mu_n$.
Furthermore, for a $K3$ surface $X$, the Brauer group is canonically isomorphic to $T(X)^\vee\otimes (\Q/\Z)$ (\cite{vGe} pp.~225).
Since symplectic automorphisms act trivially on $T(X)$, they act trivially on $\mathrm{Br}(X)$.
By \eqref{Brauer}, the statement holds.
\end{proof}

Note that the torsion part of $J(T(X)^\vee)$ in \eqref{transreg} is $T(X)^\vee \otimes (\Q/\Z) \simeq \mathrm{Br}(X)$. 
Consider the following diagram.
\begin{equation}\label{roitman}
\begin{tikzcd}
\CH^2(X,1)_{\ind} \arrow[rrr,"r"] &&&  J(T(X)^\vee)  \\
(\CH^2(X,1)_{\ind})_{\mathrm{tor}} \arrow[u,hook] & \mathrm{Br}(X)(1) \arrow[l,"\eqref{Brauer}"',"\sim"]  \arrow[r,"\sim"] &  \mathrm{Br}(X) \arrow[r,"\sim"]& T(X)^\vee \otimes (\Q/\Z) \arrow[u,hook]
\end{tikzcd}
\end{equation}
The commutativity of \eqref{roitman} implies that the transcendental regulator induces an isomorphism between torsion parts.
However, the author does not know the commutativity of \eqref{roitman} follows from Kahn's construction of the isomorphism \eqref{Brauer}.
Nevertheless, if we assume such Roitman-type result, we can say the following.
\begin{prop}\label{rot}
If the map \eqref{transreg} induces an isomorphism between torsion parts, then Conjecture \ref{main} implies that $\Aut_s(X)$ acts trivially on $\CH^2(X,1)_\ind$.
\end{prop}
\begin{proof}
Let $\xi \in \CH^2(X,1)_\ind$ and $\sigma\in\Aut_s(X)$.
If Conjecture \ref{main} is true, $\sigma^*(\xi) - \xi$ is torsion. 
However, since $\sigma$ acts trivially on the target of \eqref{transreg}, the element $\sigma^*(\xi)-\xi$ lies in the kernel of \eqref{transreg}.
By the assumption, we conclude $\sigma^*(\xi)-\xi=0$.
\end{proof}

\bibliographystyle{plain}
\bibliography{reference}

\end{document}